\documentclass{amsart}[]

\usepackage[latin1]{inputenc}
\usepackage{amsfonts}
\usepackage{amsmath}
\usepackage{amsthm}
\usepackage{amssymb}
\usepackage{latexsym}
\usepackage{enumerate}
\usepackage{comment} 
\usepackage{hyperref}
\usepackage[all]{xy}
\usepackage{xcolor}
\usepackage[square,numbers]{natbib}
\usepackage{lineno}

\theoremstyle{plain}

\newtheorem{thm}{Theorem}[section]
\newtheorem{cor}[thm]{Corollary}
\newtheorem{lem}[thm]{Lemma}
\newtheorem{prop}[thm]{Proposition}
\newtheorem{quest}[thm]{Question}
\theoremstyle{definition}

\newtheorem{defn}{Definition}
\newtheorem{rem}{Remark}

\numberwithin{equation}{section}

\newcommand{\A}{\mathcal{A}}
\newcommand{\cc}{C$^*_0$-polyhedron }

\newcommand{\B}{\mathbb{B}}
\newcommand{\K}{\mathbb{K}}
\newcommand{\M}{\mathbb{M}}
\newcommand{\C}{\mathbb{C}}

\begin{document}

\title{C*-convexity and C*-Polyhedron}
\author{Clayton Suguio Hida}
\address{Clayton Suguio Hida, Faculdade de Tecnologia de S\~ao Paulo - FATEC}
\email{clayton.hida@fatec.sp.gov.br}

\begin{abstract}
A polyhedron in a Banach space is a family of points $\mathcal{X}$ such that for every $x\in \mathcal{X}$, there is a closed convex set $C$ such that $a\notin C$ and $\mathcal{X}\setminus\{x\}\subset  C$.  In this article, we consider the notion of C*-convexity and introduce the notion of a C*-polyhedron, which is a noncommutative version of the notion of polyhedrons. We investigate the largest size of a C*-polyhedron in some classical C*-algebras.

\end{abstract}

\maketitle
\tableofcontents

\section{Introduction}

The notion of a polyhedron in a Banach space was introduced in \cite{granero2003kunen}:
\begin{defn}Let $X$ be a Banach space. A family $(x_\alpha)_\alpha$ is a polyhedron if for every $\alpha<\kappa$, there is a closed convex set $C$ such that $a_\alpha \notin C$ and $a_\beta \in C$ for every $\alpha\neq \beta$.
\end{defn}

Since a polyhedron is a discrete set in the norm topology, it follows that the size of a polyhedron is bounded by the density of the Banach space. In particular, separable Banach spaces admit only countable polyhedrons (see Section 2 of \cite{granero2003kunen}). 

In the case of nonseparable Banach spaces, the problem of the existence of an uncountable polyhedron has a set-theoretic flavour: Assuming the Continuum hypothesis (CH), the Kunen line (see \cite{negrepontis}) gives a consistent example of a nonseparable Banach space of the form $C(K)$ without an uncountable polyhedron (see \cite{negrepontis}). On the other hand, assuming the Martin Maximum, S. Todorcevic \cite{quocientstevo} proved that every nonseparable Banach space has an uncountable biorthogonal system. Since every uncountable biorthogonal system generates an uncountable polyhedron (see Proposition \ref{ubabs_polyhedron}), it follows that it is consistent that every nonseparable Banach space has an uncountable polyhedron. In particular, the question of whether every nonseparable Banach space has an uncountable polyhedron is undecidable in ZFC.

In this article, we consider the following noncommutative versions of the notion of convex sets and absolutely convex sets in linear spaces: 

A subset $\mathcal{X}$ of a unital C*-algebra $\A$ is:
\begin{enumerate}
\item \textbf{C*-convex} if $\sum_{i=1}^n A_i^* x_i A_i \in \mathcal{X}$, whenever  $\{x_1, \cdots, x_n\}\subset \mathcal{X}$ and \\
$\{A_1, \cdots, A_n\}\subset \A$ with $\sum_{i=1}^n A_i^* A_i = 1$ .
\item \textbf{C*-absolutely convex}  if $\sum_{i=1}^n A_i^* x_i A_i \in \mathcal{X}$, whenever  $\{x_1, \cdots, x_n\}\subset \mathcal{X}$ and $\{A_1, \cdots, A_n\}\subset \A$ with $\sum_{i=1}^n A_i^* A_i \leq  1$ .
\end{enumerate}

We introduce the following noncommutative convex families:

\begin{defn} Let $\A$ be a unital C*-algebra. A family $(a_\alpha)_{\alpha<\kappa}$ is a C*-polyhedron (\cc) of size $\kappa$ if for every $\alpha<\kappa$, there is a closed C*-convex (C*-absolutely convex) set $C$ such that $a_\alpha \notin C$ and $a_\beta \in C$ for every $\alpha\neq \beta$.
\end{defn}

The main purpose of this paper is to investigate the existence of large C*-polyhedrons in nonseparable C*-algebras. In particular, we would like to see if the same set-theoretic aspects of polyhedron in nonseparable Banach spaces appear in the case of C*-polyhedrons and nonseparable C*-algebras. 

In Section \ref{preliminares}, we review some basic results and definitions of polyhedrons in Banach spaces and the notion of C*-convexity in C*-algebras.
In Section \ref{c-polyhedron}, we define and study some basic properties of C*-polyhedrons. The main result of this section is the fact that every C*-algebra has an infinite \cc(see Theorem \ref{c_infinity_poly}). In Section \ref{sufficient_necessary} we present certain sufficient and necessary conditions for the existence of a C*-polyhedron. In particular, we use the noncommutative Hanh-Banach version of B. Magajna (\cite{magajna2000c}) to associate every C*-polyhedron with a system of homomorphism (see Theorem \ref{th_polyhedron_c}).

Section \ref{sec_commutative} is dedicated to the existence of C*-polyhedron in commutative C*-algebras. We summarize our main findings in this section in the following theorem:

\begin{thm}Let $\A = C(K)$ be a unital commutative C*-algebra.
\begin{enumerate}
\item $\A$ has a \cc of size $\kappa$ if and only if $\A$ has a C*-polyhedron of size $\kappa$.
\item If $K$ is hereditarily Lindel\"of, then $\A = C(K)$ does not have an uncountable \cc of projections.
\item If $K$ is not hereditarily Lindel\"of, then 
\begin{enumerate}
\item Under CH, there is a nonmetrizable scattered compact Hausdorff space $K$, such that $C(K)$ is not hereditarily Lindeloff and $C(K)$ does not have an uncountable C*-polyhedron.
\item Under PFA, every $C(K)$ with $K$ nonmetrizable (0-dimensional) compact Hausdorff space has an uncountable noncommutative C*-polyhedron (of projection).
\end{enumerate}
\end{enumerate}

\end{thm}
\begin{proof}
See Proposition \ref{equivalent_ck} and Theorem \ref{main_ck}.
\end{proof}

In the last section, we investigate the existence of uncountable C*-polyhedrons in the classical C*-algebras of compact operators $\K(H)$ and in $\B(H)$. We prove the following result:

\begin{thm} $ $
\begin{enumerate}
\item Let $H$ be a Hilbert space. Then $\K(H)$ does not have an uncountable \cc.
\item The C*-algebra $\B(\ell_2)$ does not have an uncountable C*-polyhedron of compact operators and it does not have an uncountable \cc of self-adjoint elements.
\end{enumerate}
\end{thm}
\begin{proof}
See Theorem \ref{non_exist_compact}, Proposition \ref{bh_non_compact} and Theorem \ref{bh_non_existence}.
\end{proof}

Notation should be standard. 
Let $\A$ be a C*-algebra. We denote by $d(\A)$ the topological density of $\A$ in the norm topology. Given $A,B\subset \A$, de distance between $A$ and $B$ will be denoted by $dist(A, B)$. 
See  \cite{kenneth1996c}, \cite{dixmier1977c}, \cite{li1992introduction},  \cite{murphy1990c} and \cite{pedersen1979c} for other definitions and results about C*-algebras.

\section{Preliminares}\label{preliminares}
In this section, we review some basic definitions and results on polyhedrons and C*-convexity that will be used in the rest of the article. 

\subsection{Polyhedrons}
The notion of a polyhedron in a Banach space\footnote{In this article, all Banach spaces are assumed to be Banach spaces over the complex numbers} was introduced in \cite{granero2003kunen}:
\begin{defn}Let $X$ be a Banach space. A family $(x_\alpha)_\alpha$ is a polyhedron if for every $\alpha<\kappa$, there is a closed convex set $C$ such that $a_\alpha \notin C$ and $a_\beta \in C$ for every $\alpha\neq \beta$.
\end{defn}

Every polyhedron in a Banach space $X$ is a discrete set in the norm topology. In particular, the size of a polyhedron is bounded above by the density $d(X)$.  Let us prove that every Banach space has an infinite polyhedron. 

\begin{lem}\label{c_infinity_poly}The one-dimensional Banach space $\mathbb{C}$ has an infinity polyhedron.
\end{lem}
\begin{proof}
Define $\lambda_0 = (1,0)$ and for each $n\geq 1$, define $\lambda_n = e^{i((\sum_{j=1}^n\frac{1}{2^j}) \frac{\pi}{2})}$. Observe that for each $n\in \mathbb{N}$, the distance between $\lambda_n$ and $co(\{\lambda_0, \cdots, \lambda_{n-1}, \lambda_{n+1}, \cdots\}\cup\{0\})$ is equal to the distance of $\lambda_n$ and the line segment $\overline{\lambda_{n-1}\lambda_{n+1}}$, which is positive. In particular, $\lambda_n \notin \overline{co}(\{\lambda_0, \cdots, \lambda_{n-1}, \lambda_{n+1}, \cdots\}\cup\{0\})$.
\end{proof}

\begin{prop}\label{infinite_polyhedron}Every Banach space has an infinite polyhedron.
\end{prop}
\begin{proof}
Observe that the notion of polyhedron is preserved by superspaces, i.e., if $Y$ is a Banach subspace of a Banach space $X$, then a polyhedron in $Y$ is also a polyhedron in $X$. From Lemma \ref{c_infinity_poly} we conclude that every Banach space has an infinite polyhedron.
\end{proof}

Recall that an absolutely convex set in a vector space $X$ is a subset $K\subset X$ such that $\sum_{i=1}^n \lambda_i x_i \in K$ whenever $x_1, \cdots, x_n\in K$ and $\sum_{i=1}^n|\lambda_i|\leq 1$. 

\begin{defn}Let $X$ be a Banach space. A family $(x_\alpha)_\alpha$ is a 0-polyhedron if for every $\alpha<\kappa$, there is an absolutely closed convex set $C$ such that $a_\alpha \notin C$ and $a_\beta \in C$ for every $\alpha\neq \beta$.
\end{defn}
Since every absolutely convex set is, in particular,  a convex set, it follows that every 0-polyhedron is, in particular,  a polyhedron.
Moreover,  if $(a_\alpha)_{\alpha<\omega_1}$ is a polyhedron in a Banach space $X$, then it is easy to see that $(a_\alpha - a_{0})_{0<\alpha<\omega_1}$ is a 0-polyhedron. 

In particular, a Banach space $X$ admits an uncountable polyhedron if and only if $X$ admits an uncountable 0-polyhedron. 

The existence of uncountable polyhedrons in a Banach space is associated with the existence of other structures (e.g convex right-separated family, $\omega$-independent family, see \cite{granero2003kunen} and \cite[Section 8.1]{hajek2007biorthogonal}).  We will be interested in the following:

\begin{defn}Let $X$ be a Banach space. A bounded family $(x_\alpha, f_\alpha)_{\alpha<\kappa}\subset X\times X^*$ is said to be an uncountable bounded almost biorthogonal system  (UBABS) if there is a real number $0\leq \eta<1$ such that $f_\alpha(x_\alpha)=1$ and $f_\alpha(x_\beta)\leq \eta$ if $\alpha\neq \beta$.
\end{defn}

\begin{prop}\label{ubabs_polyhedron}Let $X$ be a Banach space. Then $X$ has an uncountable polyhedron if and only if $X$ has an uncountable UBABS.
\end{prop}
\begin{proof}
See \cite[Proposition 2.2]{granero2003kunen}
\end{proof}

The reader is referred to \cite{granero2003kunen} and \cite[Section 8.1]{hajek2007biorthogonal} for other results and definitions on polyhedrons.
% 12, 13, 17, 19, 20
\subsection{C*-convexity.}

The notion of C*-convexity (also known as matrix-convexity in matrix algebras) was formally introduced and studied in \cite{loebl1981some}. It is a noncommutative version of the notion of convex sets in linear spaces and has been the subject of many studies, where some notions related to convexity (e.g., extreme points, noncommutative Hahn-Banach versions, Choquet theory) have been generalized in the noncommutative context. The reader is referred to \cite{davidson2015choquet}, \cite{davidson2019noncommutative}, \cite{effros1997matrix}, \cite{hopenwasser1981c}, \cite{kleski2014boundary}, \cite{magajna2000c}  for further results and applications of C*-convexity.

\begin{defn}[\cite{loebl1981some}]
Let $\A$ be a unital C*-algebra. A set $\mathcal{X}\subset \A$ is C*-convex if  $\sum_{i=1}^n A_i^* x_i A_i \in \mathcal{X}$, whenever  $\{x_1, \cdots, x_n\}\subset \mathcal{X}$ and $\{A_1, \cdots, A_n\}\subset \A$ with $\sum_{i=1}^n A_i^* A_i = 1$ .
\end{defn}

For instance, the set of positive elements $P = \{T: 0\leq T\leq 1\}$ is a C*-convex set in any C*-algebra (see \cite[Example 1]{loebl1981some}). Another important class of examples comes from matricial ranges: for every operator $T\in \B(H)$, the n-th matricial range $W^n(T)$ is a C*-convex set (see \cite{narcowich1981toeplitz}).

If $\mathcal{X}$ is C*-convex, it is convex in the usual sense. In general, the reverse implication does not hold:

\begin{rem}\label{remark1}Let $\mathcal{X}\subset \M_2(\C)$ be the single point set $\mathcal{X} = \{A\}$ where 
$A = \begin{pmatrix}
		2 & 1\\
		1 & 1
	 \end{pmatrix}$.
Since $\mathcal{X}$ is a one-point set, it is convex. On the other hand, if $U = \frac{1}{\sqrt{2}}\begin{pmatrix}
							1 & -1\\
							1 & 1
	 						\end{pmatrix}$, 
then $U$ is unitary and $U^* A U = \frac{1}{\sqrt{2}}\begin{pmatrix}
														5 & -1\\
														1 & -1
	 													\end{pmatrix} \notin \mathcal{X}$.
This shows that $\mathcal{X}$ is not C*-convex, although it is convex in the usual sense.	 																		
\end{rem}

\begin{lem}Let $\A$ be an unital C*-algebra and $\mathcal{X}\subset \A$:
\begin{enumerate}
\item If $\mathcal{X}$ is C*-convex and $A\in \mathcal{X}$, then $\{V^* A V: V^*V =1\}\subset \mathcal{X}$. In particular, if $\mathcal{X}$ is C*-convex, and $L$ is unitarily equivalent to $A\in \mathcal{X}$, then $L\in \mathcal{X}$.
\item Let $\lambda \in \mathbb{C}$ be a scalar. Then,  $\mathcal{X}+ \lambda I$ is C*-convex if and only if  $\mathcal{X}$ is. In particular, if $\lambda \in \mathcal{X}$, we can assume that in most cases, $0\in \mathcal{X}$.
\end{enumerate}
\end{lem}

\begin{proof}
See Remark 2 and Remark 4 of \cite{loebl1981some}.
\end{proof}

\begin{defn}Let $\A$ be a C*-algebra. If $S\subset \A$, let $MCL(S)$ denote the smallest norm closed C*-convex set containing $S$.
\end{defn}

\begin{lem}\label{closure_convex}Let $\A$ be a C*-algebra and $\mathcal{X}\subset \A$. If $\mathcal{X}$ is C*-convex, so is its norm closure.
\end{lem}
\begin{proof}
See Lemma 4 of \cite{loebl1981some}. 
\end{proof}

%\begin{proof}
%Let $T_1, \cdots, T_n \in \overline{K}$ and $A_1, \cdots, A_n$ such that $\sum_{i=1} A_i^* A_i  =1$.
%\begin{enumerate}
%\item Since $\sum_{i=1} A_i^* A_i  =1$ we have that $A_i^* A_i\leq 1$ and therefore $\|A_i\|\leq 1$.
%\item Let $\varepsilon>0$ and consider $S_i \in K$ such that 
%$$\|T_i - S_i\|\leq \varepsilon/n$$
%\item Since $S_i \in K$, we have that $\sum_i A_i^* S_i A_i \in K$.
%\item 
%$$\|\sum_i A_i^* T_i A_i  - \sum_i A_i^* S_i A_i \| = \|\sum_i A_i^* (T_i - S_i )A_i \|\leq \sum_i \|A_i^* (T_i - S_i )A_i\|%<\varepsilon$$
%\end{enumerate}
%\end{proof}

\begin{lem}Let $\A$ be a C*-algebra and $S\subset \A$. Then
$$MCL(S) = \overline{\{\sum_i^n A_i^* X_i A_i, \sum_i^n A_i^* A_i = 1, X_i\in S\}}$$
\end{lem}
\begin{proof}
Define $\mathcal{X} = \{\sum_i^n A_i^* X_i A_i, \sum_i^n A_i^* A_i = 1, X_i\in S\}$. 
Then $\mathcal{X}$ is C*-convex.
In fact, let $\{Y_1, \cdots, Y_n\}\subset \mathcal{X}$ and $\{A_1, \cdots, A_n\}\subset \A$ with $\sum_{i=1}^n A_i^* A_i = 1$.
We have that $Y_i = \sum_j^{n_i} A_{i,j}^* X_{i,j} A_{i,j}$. Then
\begin{align*}
\sum_{i=1}^n A_i^* Y_i A_i & =  \sum_{i=1}^n A_i^* (\sum_j^{n_i} A_{i,j}^* X_{i,j} A_{i,j})A_i \\
							& = \sum_{i=1}^n \sum_j^{n_i} A_i^* A_{i,j}^* X_{i,j} A_{i,j}A_i\\
							& =  \sum_{i=1}^n \sum_j^{n_i} (A_{i,j} A_i)^* X_{i,j} (A_{i,j}A_i)\\
\end{align*}
and 
\begin{align*}
\sum_{i=1}^n \sum_j^{n_i} (A_{i,j} A_i)^* (A_{i,j}A_i) & =  \sum_{i=1}^n \sum_j^{n_i} A_i^* A_{i,j}^* (A_{i,j}A_i)\\
& =  \sum_{i=1}^n  A_i^* (\sum_j^{n_i} A_{i,j}^* A_{i,j})A_i \\
& =  \sum_{i=1}^n  A_i^* A_i  = 1.
\end{align*}
This proves that $\mathcal{X}$ is C*-convex.

Since $S\subset MCL(S)$, this follows that $\mathcal{X}\subset MCL(S)$. In particular, $\overline{\mathcal{X}}\subset MCL(S)$. 

On the other hand, $\mathcal{X}$ is C*-convex, and thefore $\overline{{X}}$ is C*-convex and closed (see Lemma \ref{closure_convex}). Since $S\subset \overline{\mathcal{X}}$, it follows that $MCL(S)\subset \overline{\mathcal{X}}$.
\end{proof}

Let's define the following noncommutative version of the notion of an absolutely convex set:

\begin{defn}
Let $\A$ be a unital C*-algebra. A set $\mathcal{X}\subset \A$ is C*-absolutely convex if $\{x_1, \cdots, x_n\}\subset \mathcal{X}$ and $\{A_1, \cdots, A_n\}\subset \A$ with $\sum_{i=1}^n A_i^* A_i \leq  1$ implies 
$${\sum_{i=1}^n A_i^* x_i A_i \in \mathcal{X}}.$$
\end{defn}

\begin{lem}Let $\A$ be a unital and $X\subset \A$. Then $MCL(X\cup\{0\})$ is a C*-absolutely convex set.
\end{lem}
\begin{proof}
Define $C: = MCL(X\cup\{0\})$. Let $x_1, \cdots, x_n \in C$ and $A_1, \cdots, A_n$ such that $\sum_{i=1}A^*_i A_i\leq 1$. Let us prove that ${\sum_{i=1}^{n} A^*_i x_i A_i\in C}$.

Define $A_{n+1} = \sqrt{1 - (\sum_{i=1}A^*_i A_i)}$ and  $x_{n+1} = 0$ . Then 
$$\sum_{i=1}^n A^*_i x_i A_i = \sum_{i=1}^{n+1} A^*_i x_i A_i.$$

Since $\sum_{i=1}^{n+1}A^*_i A_i = 1$, it follows that 
$\sum_{i=1}^{n+1} A^*_i x_i A_i\in C$ and therefore ${\sum_{i=1}^{n} A^*_i x_i A_i\in C}$.
%If $x_1 = 0$, do the same with $\sum_{i=2}A^*_i A_i\leq 1$

\end{proof}

\section{C* - polyhedron and \cc}\label{c-polyhedron}
In this section, we define the main objects of this article and we explore some of its main properties.

\begin{defn}
Let $\A$ be a unital C*-algebra. A family $(a_\alpha)_{\alpha<\kappa}\subset \A$ is a C*-polyhedron of size $\kappa$ if for every $\alpha<\kappa$
$$a_\alpha \notin  MCL(\{a_\beta: \beta\neq \alpha\})$$
\end{defn}

Since C*-convex set is in particular convex, it follows that C*-polyhedron implies polyhedron in the classical sense. The following remark shows that these two notions are not equivalent:

\begin{rem}\label{remark1}
Let $\A  = M_2(\mathbb{C})$ and $K = \{P_1, P_2\}$ where 
$P_1 = \begin{pmatrix}
		  		1 & 0\\
		  		0 & 0
	 	 		\end{pmatrix}$ and $P_2 = \begin{pmatrix}
		  		0 & 0\\
		  		0 & 1
	 	 		\end{pmatrix}$.
	 	 		
Then $K$ is a polyhedron. On the other hand, if $U  = \begin{pmatrix}
		  		0 & 1\\
		  		-1 & 0
	 	 		\end{pmatrix}$ 
then  
	 	 		$$U^* P_1 U =  \begin{pmatrix}
		  		0 & -1\\
		  		1 & 0
	 	 		\end{pmatrix} \begin{pmatrix}
		  		1 & 0\\
		  		0 & 0
	 	 		\end{pmatrix} \begin{pmatrix}
		  		0 & 1\\
		  		-1 & 0
	 	 		\end{pmatrix}= \begin{pmatrix}
		  					0 & 0\\
		  					0 & 1
	 	 					\end{pmatrix} = P_2$$
In particular, $P_2\in MCL(\{P_1\})$ which shows us that $K$ is not a C*-polyhedron.

\end{rem}

%\begin{lem}For every $n\in \mathbb{N}$, the one-dimensional C*-algebra $\mathbb{C}$ has a C*-polyhedron of size $n$.
%\end{lem}

%\begin{proof}
%\begin{center}
%\includegraphics[scale=0.4]{img/n-poly_complex_numbers.png}
%\end{center}
%\end{proof}

We introduce now the following variant of the notion of a C*-polyhedron:
 
\begin{defn}Let $\A$ be a C*-algebra. A family $(a_\alpha)_{\alpha<\kappa}$ is a \cc of size $\kappa$ if for every $\alpha<\kappa$
$$a_\alpha \notin  MCL(\{a_\beta: \beta\neq \alpha\}\cup\{0\})$$
\end{defn}

The notion of \cc is related to the notion of C*-absolutely convex sets:
\begin{lem}
A family $(a_\alpha)_{\alpha<\kappa}$ is a \cc if and only if for every ${\alpha<\kappa}$, there is a C*-absolutely convex set $C$ such that $a_\alpha \notin C$ and $a_\beta \in C$ for every $\beta\neq \alpha$.
\end{lem}
\begin{proof}
It follows from the fact that $MCL(\{a_\beta: \beta\neq \alpha\}\cup\{0\})$ is a C*-absolutely convex set.
% proving that C is C*-absoutely convex set: 
%let x_1, \cdots, x_n \in C$ and $A_1, \cdots, A_n$ such that $\sum_{i=1}A^*_i A_i\leq 1$. Then define $A_{n+1} = \sqrt{1 - (\sum_{i=1}A^*_i A_i)}$ $x_{n+1} = 0$ . Then $\sum_{i=1}^n A^*_i x_i A_i = \sum_{i=1}^{n+1} A^*_i x_i A_i\in C$. If $x_1 = 0$, do the same with $\sum_{i=2}A^*_i A_i\leq 1$
\end{proof}

Every \cc is in particular a C*-polyhedron, but these two notions are not the same. For instance, consider the one-dimensional C*-algebra $\mathbb{C}$ and $\mathcal{F} = \{(1,0), (0,1), (\frac{1}{3}, \frac{1}{3})\}\subset \mathbb{C}$. Then $(\frac{1}{3}, \frac{1}{3}) = \frac{1}{\sqrt{3}}(1,0)\frac{1}{\sqrt{3}} + \frac{1}{\sqrt{3}}(0,1)\frac{1}{\sqrt{3}}$, with $(\frac{1}{\sqrt{3}})^2 + (\frac{1}{\sqrt{3}})^2 = \frac{2}{3}<1$ which shows that $\mathcal{F}$ is not a \cc. On the other hand, if $a,b\in \mathbb{C}$ and $\|a\|^2 + \|b\|^2 = 1$, then $a(1,0)a + b(0,1)b \neq (\frac{1}{3}, \frac{1}{3}) $. In particular, $\mathcal{F}$ is a C*-polyhedron.

We have seen in Section \ref{preliminares} that a Banach space $X$ has an uncountable polyhedron if, and only if, $X$ has an uncountable 0-polyhedron. We will prove in Proposition \ref{equivalent_ck} that the same happens with C*-polyhedron and \cc in commutative C*-algebras, but we do not know what is the situation in general C*-algebras.

In the one-dimensional C*-algebra, the notion of C*-polyhedron and the notion of polyhedron coincide. In particular, by Lemma \ref{c_infinity_poly}, the one-dimensional C*-algebra $\mathbb{C}$ has an infinity \cc (in particular, it has an infinite C*-polyhedron).

Since every C*-algebra has the one-dimensional C*-algebra $\mathbb{C}$ as a subalgebra, using the same idea as in Proposition \ref{infinite_polyhedron}, we could conclude that every C*-algebra has an infinite \cc. The problem is that notion of \cc is not preserved by superalgebras:

\begin{rem}
Consider the two orthogonal projections $P_1, P_2$ as in Remark \ref{remark1}.
Define  $\A = C^*(P_1, P_2)  = span(\{P_1, P_2\})\subset M_2(\mathbb{C})$. Observe that
$$\begin{pmatrix}
\lambda_1 & 0\\
0 		  & \lambda_2
\end{pmatrix} * 
\begin{pmatrix}
1 & 0\\
0 & 0
\end{pmatrix} * 
\begin{pmatrix}
\overline{\lambda_1} & 0\\
        0 & \overline{\lambda_2}
\end{pmatrix} = 
\begin{pmatrix}
|\lambda_1|^2 & 0\\
        0 & 0
\end{pmatrix}
$$
In particular, $P_2\notin MCL(\{P_1, 0\})$. The same holds with $P_1$ replaced by $P_2$.
This shows us that $\{P_1, P_2\}$ is a \cc in $\A$.
On the other hand, as we have seen in Remark \ref{remark1}, $\{P_1, P_2\}$ is not a C*-polyhedron (and therefore, it is not a \cc) in $M_2$. 
\end{rem}

However, with a different approach, we are able to prove the following:

\begin{thm}\label{infinite}Let $\A$ be an unital C*-algebra. Then $\A$ has an infinity \cc.
\end{thm}

\begin{proof}
Let $(\lambda_n)_n$ be the sequence of complex numbers as in Lemma \ref{c_infinity_poly} and denote by $I$ the unity of $\A$. Define
$X_n = \lambda_n I$. We claim that $(X_n)_n$ is a \cc.

In fact, consider $X_{n_0}$ and $X_{n_1}, \cdots, X_{n_k}$. Let $A_1, \cdots , A_n\in \A$ such that ${\sum A_i^* A_i \leq  1}$. Consider $\tau$ a state on $\A$.

Then $$|\tau(X_{n_0} - \sum_{j=1}^kA_j^* X_{n_j} A_j)|\leq |X_{n_0} - \sum_{j=1}^kA_j^* X_{n_j} A_j|$$

On the other hand, $\tau(X_{n_0} - \sum_{j=1}^kA_j^* X_{n_j} A_j) = \lambda_{n_0} - \sum_{j=1}^k \lambda_{n_j}\tau(A_j^*A_j)$.

Since $A_i^*A_i$ is positive and $\sum_{i=1}^k \tau(A_i^*A_i) \leq 1$, it follows that $\sum_{j=1}^k \lambda_{n_j}\tau(A_j^*A_j)$ is a linear absolutely convex combinations of the elements $(\lambda_{n_j})_j$. In particular, by the previous lemma, there is $\varepsilon_n>0$ such that  $0<\varepsilon_{n_0}\leq |\lambda_{n_0} - \sum_{j=1}^k \lambda_{n_j}\tau(A_j^*A_j)|$. Then
$$0<\varepsilon_{n_0}< |X_{n_0} - \sum_{j=1}^kA_j^* X_{n_j} A_j|$$
This shows that $X_{n_0}\notin MCL(\{X_n: n\neq n_0\}\cup \{0\})$.

\end{proof}

By Theorem \ref{infinite}, every C*-algebra $\A$ has an infinite \cc. In particular, the cardinal $\omega$ is a lower bound for the supremum of all sizes of a \cc in a C*-algebra. The following lemma gives us an upper bound:

\begin{lem}\label{upperbound}Let $\A$ be a C*-algebra. Then $d(\A)$ is an upper bound for the size of a C*-polyhedron.
\end{lem}
\begin{proof}
This follows from the fact that a C*-polyhedron is in particular a discrete set in the norm topology. 
\end{proof}

In particular, if $\A$ is a separable C*-algebra, then $\A$ has only infinite countable C*-polyhedrons and we can ask the following question:

\begin{quest}\label{question_existence}Is it true that every nonseparable C*-algebra has an uncountable \cc (or C*-polyhedron) ?
\end{quest}

In sections \ref{sec_commutative} and \ref{sec_noncommutative} we give some partial answers to this question.

We conclude this section with some results on uncountable \cc that will be used later.

\begin{lem}Let $\A$ be a C*-algebra and suppose $\A$ has an uncountable \cc (C*-polyhedron). Then $\A$ has a bounded uncountable \cc (C*-polyhedron).
\end{lem}
\begin{proof}
Let $(a_\alpha)_{\alpha<\omega}$ be an uncountable C*-polyhedron. Since 
$$(a_\alpha)_{\alpha<\omega} \subset \bigcup_{n = 1}^\infty \{a: \|a\|\leq n\}$$
there is $n$ such that $\Gamma = \{\alpha<\omega_1, \|a_\alpha\|<n\}$ is uncountable.

Define $b_\alpha = a_{\alpha}/n$ for each $\alpha\in \Gamma$. Then $(b_\alpha)_{\alpha\in \Gamma}$ is an uncountable and bounded C*-polyhedron.
\end{proof}

The following lemma will allow us to replace elements by elements in a dense set:

\begin{lem}\label{replace_elements}Let $\A$ be a C*-algebra. Suppose $(a_\alpha)_{\alpha<\omega_1}$ is an uncountable \cc (C*-polyhedron). Then there is an uncountable index $\Gamma\subset \omega_1$ and $\varepsilon>0$ such that if $\alpha \in \Gamma$ and  $\|a_\alpha - b\|\leq \varepsilon$, then $\{a_\beta: \beta\neq \alpha \in \Gamma\}\cup \{b\}$ is a \cc (C*-polyhedron).
\end{lem}
\begin{proof}
Suppose $(a_\alpha)_{\alpha<\omega_1}$ is a C*-polyhedron. The proof of the \cc case is similar.
Since $(a_\alpha)_{\alpha<\omega_1}$ is an uncountable C*-polyhedron, for each $\alpha<\omega_1$, there is $n_\alpha$ such that ${0<\frac{1}{n_\alpha}<dist(a_\alpha, MCL(\{a_\beta: \beta\neq \alpha\}))}$.

Consider $\Gamma \subset \omega_1$ and $n\in \mathbb{N}$ such that $n_\alpha = n$ for each $\alpha \in \Gamma$.
Define $\varepsilon = \frac{1}{2n}$. To make the notion clear, let's keep the notation $(a_\alpha)_{\alpha<\omega_1}$ for the subsequence.
Fix $\alpha<\omega_1$ and consider $b\in \A$ such that $\|a_\alpha - b\|\leq \varepsilon$.

Suppose $\{a_\beta: \beta \neq \alpha\}\cup \{b\}$ is not a C*-polyhedron. Then either $b{\in MCL(\{a_\beta: \beta \neq \alpha\})}$ or there is $\beta_0 \neq \alpha$ such that $a_{\beta_0}\in MCL(\{a_\beta: \beta \neq \alpha, \beta_0\}\cup \{b\})$.

\begin{enumerate}
\item Suppose $b\in MCL(\{a_\beta: \beta \neq \alpha\})$. Then there are $A_1, \cdots, A_n$ and $a_{\beta_1}, \cdots, a_{\beta_n}$ with $\sum_i A_i^* A_i= 1$ such that
$$\|b - \sum_i A_i^* a_{\beta_i} A_i\|\leq \varepsilon.$$
Since $\|a_\alpha - b\|\leq \varepsilon$, we would conclude that  $\|a_\alpha - \sum_i A_i^* a_{\beta_i} A_i\|\leq 2\varepsilon$, a contradiction with the choice of $\varepsilon$.
\item Suppose there is $\beta_0 \neq \alpha$ such that $a_{\beta_0}\in MCL(\{a_\beta: \beta \neq \alpha, \beta_0\}\cup \{b\})$. Then there are $A_1, \cdots, A_n, A$ and $a_{\beta_1}, \cdots, a_{\beta_n}$ with $\beta_1, \cdots, \beta_n \neq \alpha, \beta_0$ and $\sum_i A_i^* A_i + A^*A = 1$ such that
$$\|a_{\beta_0} - \sum_i A_i^* a_{\beta_i} A_i + A^* b A\|\leq \varepsilon.$$
Then 
\begin{align*}
\|a_{\beta_0} - \sum_i A_i^* a_{\beta_i} A_i + A^* a_\alpha A\| & =  
\|a_{\beta_0} - \sum_i A_i^* a_{\beta_i} A_i + A^* b A + A^* a_\alpha A - A^* b A \|\\
&\leq  \|a_{\beta_0} - \sum_i A_i^* a_{\beta_i} A_i + A^* b A\| +\|A^* a_\alpha A - A^* b A \|\\
&\leq   \varepsilon + \|A^*\|\|a_\alpha - b\|\|A\|\leq 2\varepsilon
\end{align*}

and this would imply $dist(a_{\beta_0}, MCL(\{a_\alpha, \alpha\neq \beta_0\})\leq 2\varepsilon $, which is again a contradiction with the choice of $\varepsilon$.
\end{enumerate}

\end{proof}

\section{C*-polyhedrons and homomorphisms}\label{sufficient_necessary}
In this section, we prove some necessary and sufficient conditions for the existence of a \cc.

\begin{lem}\label{sufficient_representations}Let $\A$ be a C*-algebra and $(a_\alpha, (\pi_\alpha, H_\alpha))_{\alpha<\kappa}$ a family of points $a_\alpha\in \A$ and representations $\pi_\alpha:\A \to B(H_\alpha)$ such that 
\begin{enumerate}
\item $\pi_\alpha(a_\alpha)\neq 0$ and 
\item $\pi_\alpha(a_\beta)= 0$ if $\alpha\neq \beta$.
\end{enumerate}
Then $(a_\alpha)_{\alpha<\kappa}$ is a \cc of size $\kappa$.
\end{lem}
\begin{proof}
Suppose that $(a_\alpha)_{\alpha<\kappa}$ is not a \cc and let's get a contradiction.
Fix $\alpha<\kappa$ such that $a_\alpha \in MCL(\{a_\beta: \beta\neq \alpha\}\cup\{0\})$.
Given $\varepsilon = \frac{\|\pi_\alpha(a_\alpha)\|}{2}$, there is $a_{\beta_1}, \cdots, a_{\beta_n}\subset \{a_\beta: \beta\neq \alpha\}\cup\{0\}$ and $A_1, \cdots, A_n$ with $\sum_{i}A_i^* A_i = 1$ and 
$$\|a_\alpha - \sum_i A_i^* a_{\beta_i} A_i\|\leq \varepsilon$$
Since $$\pi_\alpha(\sum_i A_i^* a_{\beta_i} A_i) = \sum_i \pi_\alpha(A_i)^* \pi_\alpha(a_{\beta_i})\pi_\alpha(A_i) = 0,$$
we have that
$$\|\pi_\alpha(a_\alpha)\| = \|\pi_\alpha(a_\alpha) - \pi_\alpha(\sum_i A_i^* a_{\beta_i} A_i)\|\leq \| \pi_\alpha(a_\alpha - \sum_i A_i^* a_{\beta_i} A_i)\|\leq \varepsilon = \frac{\|\pi_\alpha(a_\alpha)\|}{2}$$
which is a contradiction with the fact that $\pi_\alpha(a_\alpha)\neq 0$.
\end{proof}

\begin{rem}
Let $\A = \oplus_{\alpha<\omega_1}K(\ell_2)$. Then $\A$ is a nonseparable C*-algebra. Fix a nonzero compact operator $T\in K(\ell_2)$ and consider the element $a_\alpha = (0, \cdots, T, 0, \cdots)$ for each $\alpha<\omega_1$. Let $p_\alpha: \A\to K(\ell_2)\subset B(\ell_2)$ be the canonical projection onto the $\alpha$-th coordinate.
Then, by Lemma \ref{sufficient_representations}, $(a_\alpha)_{\alpha<\omega_1}$ is an uncountable \cc in $\A$.
\end{rem}

We have seen that the notion of polyhedron is associated with the existence of a particular family of functionals (see Proposition \ref{ubabs_polyhedron}).  We will prove that a similar result holds in the case of \cc.

\begin{defn}Let $\A$ be a C*-algebra. Consider $\omega$ a state on $\A$ and denote by $\pi_\omega: \A \to H_\omega$ the GNS-representation associated with $\omega$. A function $\varphi : \A \to B(H_\omega)$ is a homomorphism of $\A$-modules if $\varphi$ is a linear map and $\varphi(axb) = \pi_\omega(a) \varphi(x) \pi_\omega(b)$ for every $a,x,b\in \A$.

\end{defn}

\begin{lem}\label{counting}Let $\A$ be a C*-algebra and $\{a_1, a_2, \cdots, a_n\}\subset \A$ . Consider a state $\omega$ and a homomorphism $\varphi : \A \to B(H_\omega)$ of $\A$-module such that\footnote{Given an element $a$ in a C*-algebra $\A$, $Re(a)$ denotes the real part of $a$, defined by $Re(a) = \frac{a + a^*}{2}$.} $Re(\varphi(a_i))\leq I$ for every $i=1, \cdots, n$. Then for every $A_1, \cdots, A_n \in \A$ with $\sum_i A_i^* A_i = 1$ we have that $Re(\varphi(\sum_i A_i^* a_i A_i))\leq I$.

\end{lem}

\begin{proof}

Fix $A_1, \cdots, A_n \in \A$ with $\sum_i A_i^* A_i = 1$. Then

\begin{align*}
I - Re(\varphi_\alpha(\sum_i A_i^* a_{i} A_i))) & =I - Re(\sum_i (\varphi_\alpha(A_i^* a_{_i} A_i)))\\
	 & = \pi(1) - Re(\sum_i ((\pi(A_i^*) \varphi_\alpha(a_{_i})\pi(A_i))))\\
     & =\sum_i ((\pi(A_i^*) I \pi(A_i))) - \sum_i ((\pi(A_i^*) Re(\varphi_\alpha(a_{i}))\pi(A_i)))\\
     & =\sum_i ((\pi(A_i^*) (I - Re(\varphi_\alpha(a_{i}))) \pi(A_i)))
\end{align*}

Let us prove that $\sum_i ((\pi(A_i^*) (I - Re(\varphi_\alpha(a_{i}))) \pi(A_i)))$ is positive. To this end, consider $\xi \in H_\omega$. Then 
\begin{align*}
\langle (I - Re(\varphi_\alpha(\sum_i A_i^* a_{\beta_i} A_i))))\xi, \xi\rangle & = \langle \sum_i ((\pi(A_i^*) (I - Re(\varphi_\alpha(a_{i}))) \pi(A_i)))\xi, \xi\rangle\\
	& = \sum_i  \langle \pi(A_i^*) (I - Re(\varphi_\alpha(a_{i}))) \pi(A_i)\xi, \xi\rangle\\
	& = \sum_i  \langle (I - Re(\varphi_\alpha(a_{i}))) \pi(A_i)\xi, \pi(A_i) \xi\rangle
\end{align*}
Since $Re(\varphi(a_i))\leq I$ for every $i=1, \cdots, n$, it follows that  
$$\langle (I - Re(\varphi_\alpha(a_{i}))) \pi(A_i)\xi, \pi(A_i) \xi\rangle\geq 0$$
for every $i=1, \cdots, n$. In particular, $\langle (I - Re(\varphi_\alpha(\sum_i A_i^* a_{\beta_i} A_i))))\xi, \xi\rangle \geq 0$. This proves that $Re(\varphi_\alpha(\sum_i A_i^* a_{\beta_i} A_i))\leq I$.

\end{proof}

\begin{prop}\label{sufficient_condition}Let $\A$ be a C*-algebra and $(a_\alpha)_{\alpha<\kappa}$ a family of points of $\A$. Suppose that for each $\alpha<\kappa$, there is a state $\omega_\alpha$ and a homomorphism of $\A$-module $\varphi_\alpha: \A \to B(H_\omega)$  such that
\begin{enumerate}
\item $Re(\varphi_\alpha(a_\beta))\leq  1$ for all $\beta \neq \alpha$;
\item $Re(\varphi_\alpha(a_\alpha))\not \leq 1$,
\end{enumerate}
Then $(a_\alpha)_{\alpha<\kappa}$ is a \cc of size $\kappa$.
\end{prop}

\begin{proof}
Suppose that $(a_\alpha)_{\alpha<\kappa}$ is not a \cc and let's get a contradiction.
Fix $\alpha<\kappa$ such that $a_\alpha \in MCL(\{a_\beta: \beta\neq \alpha\}\cup\{0\})$.
Since $Re(\varphi_\alpha(a_\alpha))\not \leq I$, there is $\xi\in H_\alpha$ such that $$\eta = \langle (I - Re(\varphi_\alpha(a_\alpha)))\xi, \xi\rangle <0$$

Given $\varepsilon < \frac{|\eta|}{2\|\varphi_\alpha\|}$, there is $a_{\beta_1}, \cdots, a_{\beta_n}\subset \{a_\beta: \beta\neq \alpha\}\cup\{0\}$ and $A_1, \cdots, A_n$ with $\sum_{i}A_i^* A_i = 1$ and 
$$\|a_\alpha - \sum_i A_i^* a_{\beta_i} A_i\|\leq \varepsilon$$

Then 
\begin{align*}
\eta  & = \langle (I - Re(\varphi_\alpha(a_\alpha)))\xi, \xi \rangle\\
	  & = \langle I\xi - Re(\varphi_\alpha(a_\alpha))\xi + Re(\varphi_\alpha(\sum_i A_i^* a_{\beta_i} A_i))\xi - Re(\varphi_\alpha(\sum_i A_i^* a_{\beta_i} A_i))\xi , \xi \rangle\\
      & = \langle (I- Re(\varphi_\alpha(\sum_i A_i^* a_{\beta_i} A_i)))\xi  + (Re(\varphi_\alpha(\sum_i A_i^* a_{\beta_i} A_i)) - Re(\varphi_\alpha(a_\alpha)))\xi  , \xi \rangle \\
	 & = \langle  (I - Re(\varphi_\alpha(\sum_i A_i^* a_{\beta_i} A_i)))\xi , \xi \rangle + 
\langle  (Re(\varphi_\alpha(\sum_i A_i^* a_{\beta_i} A_i - a_\alpha)))\xi, \xi \rangle
\end{align*}

Observe that 
$$|\langle  Re(\varphi_\alpha(\sum_i A_i^* a_{\beta_i} A_i - a_\alpha))\xi, \xi \rangle| \leq \|\varphi_\alpha\|\|\sum_i A_i^* a_{\beta_i} A_i - a_\alpha\|\leq |\eta|/2$$
In particular, we should have $$\langle  (I - Re(\varphi_\alpha(\sum_i A_i^* a_{\beta_i} A_i)))\xi , \xi \rangle<0$$
and therefore $Re(\varphi_\alpha(\sum_i A_i^* a_{\beta_i} A_i)))\not \leq I$, which is a contradiction with Lemma \ref{counting}.

\end{proof}

The following noncommutative version of the Hahn-Banach theorem is due to B. Magajna: 

\begin{thm}[\cite{magajna2000c, magajna2016c}]\label{magajna2000c}Let $K$ be a C*-convex set of $\A$ such that $0\in K$ and $x_0 \in \A\setminus K$. Then there exists a state $\omega$ and a homomorphism of $\A$-module $\varphi: \A \to B(H_\omega)$ such that
\begin{enumerate}
\item $Re(\varphi(x))\leq  1$ for all $x\in K$;
\item $Re(\varphi(x_0))\not \leq 1$,
\end{enumerate}
\end{thm}

\begin{thm}\label{th_polyhedron_c}Let $(a_\alpha)_{\alpha<\kappa}$ a family of points in $\A$. Then $(a_\alpha)_{\alpha<\kappa}$is a \cc of size $\kappa$ if and only if for each $\alpha<\kappa$, there exists a state $\omega_\alpha$ and a homomorphism of $\A$-module $\varphi_\alpha: \A \to B(H_{\omega_\alpha})$ such that
\begin{enumerate}
\item $Re(\varphi_\alpha(a_\beta))\leq  1$ for all $\beta\neq \alpha$;
\item $Re(\varphi_\alpha(a_\alpha))\not \leq 1$,
\end{enumerate}
\end{thm}

\begin{proof}
It follows from Proposition \ref{sufficient_condition} and Theorem \ref{magajna2000c}. 
\end{proof}

\section{Commutative C*-algebras and C*-polyhedrons}\label{sec_commutative}

In this section, we prove some results on unital commutative C*-algebra, i.e., C*-algebras of the form $C(K)$, for $K$ a compact Hausdorff space and the existence of C*-polyhedrons.

Let's begin with a result that gives a sufficient condition for the existence of a C*-polyhedron:

\begin{lem}\label{existence_ubabs_commu}Let $K$ be a infinite compact Hausdorff space and suppose $(f_\alpha, \delta_{x_\alpha})_{\alpha<\kappa}$ is a UBABS of type $\eta$, with $\eta<1$. Then $(f_\alpha)_{\alpha<\kappa}$ is a \cc.
\end{lem}
\begin{proof}
Fix $\alpha<\kappa$ and $\beta_1, \cdots, \beta_n <\kappa$ with $\beta_i\neq \alpha$ for each $1\leq i \leq n$. Let $g_1, \cdots, g_n \in C(K)$ such that $\sum_{i=1}^n g_i^* g_i \leq 1$. Then 

\begin{align*}
\|f_\alpha - \sum_{i=1}^n g_i^* f_{\beta_i} g_i\| & \geq   |f_\alpha(x_\alpha) - \sum_{i=1}^n (g_i^* f_{\beta_i} g_i) (x_\alpha)|\\
												  & \geq   |f_\alpha(x_\alpha)| - |\sum_{i=1}^n g_i^* f_{\beta_i} g_i (x_\alpha)|\\
												  & \geq   1 - \eta \sum_{i=1}^n|g_i(x_\alpha)|^2\\
												  & \geq 1 - \eta
\end{align*}
This shows that $dist(f_\alpha, MCL(\{f_\beta: \beta\neq \alpha\}))\geq 1 - \eta>0$ and therefore 
$${f_\alpha \notin MCL(\{f_\beta: \beta\neq \alpha\})}.$$
\end{proof}

\begin{cor}\label{discrete_ck_poly}Let $K$ be a compact Hausdorff space and suppose $K$ has a discrete set of size $\kappa$. Then $C(K)$ has a \cc of size $\kappa$. In particular, every compact Hausdorff space of weight bigger than continuum has an uncountable \cc.
\end{cor}
\begin{proof}
Suppose $(x_\alpha)_{\alpha<\kappa}$ is a discrete set in $K$ witnessed by the family of open sets $(O_\alpha)_{\alpha<\kappa}$, i.e., $x_\beta \in O_\alpha$ if and only if, $\alpha=\beta$.
By the Urysohn lemma, consider for each $\alpha<\kappa$ a function $f_\alpha:K\to [0,1]$ such that $f_\alpha(x_\alpha) = 1$ and $f_\alpha = 0$ outside $O_\alpha$. Then $(f_\alpha,\delta_{x_\alpha})_{\alpha<\kappa}$ is a UBABS of type $\eta = 0$. By Lemma \ref{existence_ubabs_commu}, it follows that $C(K)$ has a C*-polyhedron of size $\kappa$.

For the second part, we use the fact\footnote{Let $X$ be a topological space. We denote by $w(X)$ the topological weight of $X$ and $s(X)$ denotes the topological spread of $X$, which is defined as the supremum of sizes of discrete sets in $X$. See \cite{handbookconjuntos} for definitions and results on cardinal functions.} that $w(K)\leq 2^{s(K)}$ for every compact Hausdorff space (see \cite[Corollary 7.7]{handbookconjuntos}). In particular, every compact Hausdorff space with $w(K)>2^\omega$ has an uncountable discrete set, and therefore, $C(K)$ has an uncountable \cc.

\end{proof}
%\begin{prop}Let $K$ be a compact Hausdorff space and $(f_\alpha)_\alpha$ a family of positive functions in $C(K)$ such that $0 \notin spec(f_\alpha) = Img (f_\alpha)$. If there are $\alpha\neq \beta$ such that $f_\alpha(x)\leq f_\beta(x)$, then $(f_\alpha)_\alpha$ is not a C*-polyhedron.

%\end{prop}
%\begin{proof}
%In fact, consider $\alpha\neq \beta$. Then $h(x) = \sqrt{\frac{f_\alpha(x)}{f_\beta(x)}}$ is an element in $C(K)$ with $\|h\|\leq 1$  and $f_\alpha = h(x)^2 f_\beta$.
%\end{proof}
\begin{rem}
Let $K_{\omega_1} = ([0, \omega_1], \tau_{ord})$ with the order topology. Then $\A = C(K_{\omega_1})$ is a nonseparable commutative C*-algebra with an uncountable discrete set. In particular, by Corollary \ref{discrete_ck_poly}, $\A$ has an uncountable C*-polyhedron.
\end{rem}

In the case of commutative C*-algebras, we will prove that the existence of a large C*-polyhedron is associated with the existence of a large \cc. Before proving this result, we need the following lemma:

\begin{lem}\label{convex_ck_linear_comb}Let $K$ be a compact Hausdorff space and $(f_\alpha)_{\alpha<\kappa}$ a C*-polyhedron in $C(K)$. Consider $g\in C(K)$ and $\lambda\neq 0$. Then $(\lambda f_\alpha - g)_{\alpha<\kappa}$ is a C*-polyhedron.
\end{lem}

\begin{proof}
For each $\alpha<\kappa$, define $g_\alpha := \lambda f_\alpha - g$. Let us prove that $(g_\alpha)_{\alpha<\kappa}$ is a C*-polyhedron.
Fix $\alpha<\kappa$ and $\beta_1, \cdots, \beta_n <\kappa$ with $\beta_i\neq \alpha$ for each $1\leq i \leq n$.
Let $h_1, \cdots, h_n \in C(K)$ such that $\sum_{i=1}^n h_i^* h_i = 1$. Then
\begin{align*}
\|g_{\alpha} - \sum_j h_j^* g_{\beta_j}h_j\| & =  \|g_{\alpha} - \sum_j h_j^*h_j g_{\beta_j}\| \\
                                             & =    \|\lambda f_{\alpha} - g - \sum_j h_j^*h_j (\lambda f_{\beta_j} - g)\| \\
 											  & =   \|\lambda f_{\alpha} - g - (\lambda \sum_j h_j^*h_j f_{\beta_j} - \sum_j h_j^*h_jg)\| \\
										  & =  |\lambda|\|f_{\alpha} - \sum_j h_j^*h_j f_{\beta_j}\|
\end{align*}
Since $(f_\alpha)_{\alpha<\kappa}$ is a C*-polyhedron, there is $\varepsilon>0$ such that ${dist(f_\alpha, MCL(\{f_\beta: \beta\neq \alpha\})>\varepsilon}$. Then 
$$\|g_{\alpha} - \sum_j h_j^* g_{\beta_j}h_j\| = |\lambda|\|f_{\alpha} - \sum_j h_j^*h_j f_{\beta_j}\|\geq  |\lambda|\varepsilon>0$$
This shows that $g_\alpha \notin MCL({g_\beta: \beta\neq \alpha})$.

\end{proof}

\begin{prop}\label{equivalent_ck}Let $K$ be a compact Hausdorff space. Then $C(K)$ has an infinite C*-polyhedron of size $\kappa$ if and only if $C(K)$ has a \cc of size $\kappa$.
\end{prop}
\begin{proof}
Every \cc is in particular a C*-polyhedron. Suppose now that $(f_\alpha)_{\alpha<\kappa}$ is a C*-polyhedron. By Lemma \ref{convex_ck_linear_comb}, $(f_\alpha - f_0)_{\alpha<\kappa}$ is also a C*-polyhedron. Since $0$ is an element of $(f_\alpha - f_0)_{\alpha<\kappa}$, it follows that $(f_\alpha - f_0)_{\alpha<\kappa, \alpha\neq 0}$ is a C*-polyhedron*.

\end{proof}

%\begin{lem}Let $K$ be a compact Hausdorff space and $(f_\alpha)_\alpha$ an  uncountable C*-polyhedron. Then there is an uncountable C*-polyhedron $(g_\alpha)_\alpha$ such that $0 \notin Img(g_\alpha)$.
%\end{lem}

%\begin{proof}
%We can assume that there is $M>0$ such that $\|f_\alpha\|< M$.

%Define $g_\alpha = M + f_\alpha$. Since $\|f_\alpha\| < M$, we have that $|g_\alpha(x)| = |M + f_\alpha(x)|\geq M - |f_\alpha(x)| > M -  M= 0$.
%\end{proof}

Now we turn our attention to the existence of an uncountable \cc in nonseparable C*-algebras of the form $C(K)$. We will relate some topological properties of $K$ and the existence of some special types of \cc in $C(K)$.

Observe that if $K$ is connected, then $C(K)$ has no nontrivial projections. In particular, $C(K)$ does not have an infinite \cc of projections. In the case of disconnected spaces, we have the following:

\begin{prop}\label{nonexistence_projections_ck}Let $K$ be a compact Hausdorff space. If $K$ is hereditarily Lindeloff, then $C(K)$ does not admit an uncountable C*-polyhedron of projections.

\end{prop}

\begin{proof}Suppose $K$ is a hereditarily Lindeloff space and let $(p_\alpha)_{\alpha<\omega_1}\subset C(K)$ be an uncountable family of projections.
For each $\alpha<\omega_1$, since  $p_\alpha$ is a projection in $C(K)$, there is a clopen set $O_\alpha$ of $K$ such that ${p_\alpha = \chi_{O_\alpha}}$ is the characteristic function on $O_\alpha$.

Define $L := \bigcup \{O_\alpha: \alpha<\omega_1\}$. Since $K$ is hereditarily Lindeloff, there is a countable subset of index $(\alpha_i)_{i\in \mathbb{N}}$ such that $L = \bigcup \{ O_{\alpha_i}: i\in \mathbb{N}\}$.

Fix $\alpha \neq \alpha_i$ for every $i\in \mathbb{N}$.  Since $O_\alpha$ is compact and $O_\alpha \subset L = \bigcup_i^\infty O_{\alpha_i}$, it follows that we can find $k\in \mathbb{N}$ such that 
$O_\alpha \subset \bigcup_i^k O_{\alpha_i}$. 
For each $1\leq i \leq k$, define $B_i = O_\alpha \cap (O_{\alpha_i}\setminus \bigcup_{j\neq i} O_{\alpha_j})$ and the projection $q_i = \chi_{B_i}$.
Then 
%$$\chi_{O_\alpha} = A_1 \chi_{O_{\alpha_1}} + \cdots +  A_k \chi_{O_{\alpha_k}} + A *0$$
$$p_{\alpha} = q_1 p_{\alpha_1}q_1  + \cdots + q_k p_{\alpha_k} q_k$$

In particular, $(p_\alpha)_{\alpha<\omega_1}$ is not a C*-polyhedron.

\end{proof}

The following remark shows an example of a C*-algebra of the form $\A = C(K)$, such that $\A$ has an uncountable polyhedron of projections, while it does not contain any uncountable C*-polyhedron of projections:

\begin{rem}[Double Arrow space]\label{double_arrow}Consider $X = 2^\omega \times \{0,1\}$ equipped with the lexicography order and let $\tau_{ord}$ the order topology on $X$. The topological space $\mathcal{K} = (X, \tau_{ord})$ is the classical Double arrow space. 
The Double arrow space $\mathcal{K}$ is a compact Hausdorff space which is hereditarily separable and hereditarily Lindel\"of. 

For each $x\in 2^\omega$, let $f_{x}$ be the characteristic function of the open interval $[(0^\omega,0),(x,1))$, where $0^\omega: \mathbb{N}\to 2$ denotes the zero function, and consider the state $\tau_x \in C(K)^*$ defined by $\tau_x = \delta_{(x,0)} - \delta_{(x,1)}$.

Then $\tau_x (f_x) = f_x(x,0) - f_x (x,1) = 1$ and $\tau_x (f_y) = f_y(x,0) - f_y (x,1) = 0$. 

In particular, $(f_x)_{x\in 2^\omega}\cup\{0\}$ is a polyhedron of projections. On the other hand, since $X$ is hereditarily Lindel\"of, it follows by Proposition \ref{nonexistence_projections_ck} that $C(X)$ does not have an uncountable C*-polyhedron of projections.

\end{rem}

\begin{thm}\label{main_ck}Let $K$ be a compact Hausdorff space.
\begin{enumerate}
\item If $K$ is hereditarily Lindeloff, then $C(K)$ does not have an uncountable C*-polyhedron of projections.
\item If $K$ is not hereditarily Lindeloff then:
\begin{enumerate}
\item Under CH, there is a nonmetrizable scattered compact Hausdorff space $K$, such that $C(K)$ is not hereditarily Lindeloff and $C(K)$ does not have an uncountable C*-polyhedron.
\item Under PFA, every $C(K)$ with $K$ nonmetrizable (0-dimensional) compact Hausdorff space has an uncountable noncommutative C*-polyhedron (of projection).
\end{enumerate}

\end{enumerate}

\end{thm}
\begin{proof}$ $
\begin{enumerate}
\item It follows from Proposition \ref{nonexistence_projections_ck}.
\item 
\begin{enumerate}
\item Under CH, consider $K$ the Kunen space (see \cite{negrepontis}). Then $K$ is nonmetrizable and $K$ is not hereditarily Lindeloff. But $C(K)$ does not have an uncountable polyhedron and therefore, it does not contain an uncountable C*-polyhedron.
\item Assuming PFA, every nonmetrizable non hereditarily Lindeloff compact Hausdorff space $K$ has an uncountable discrete set  (see \cite[Theorem 8.9]{todorcevic1989partition}). Then the result follows from Corollary \ref{discrete_ck_poly}.
\end{enumerate}
\end{enumerate}
\end{proof}

\begin{cor}The question whether there is an uncountable C*-polyhedron in $C(K)$, where  $K$ is a 0-dimensional nonmetrizable non hereditarily Lindeloff compact Hausdorff space, is independent from ZFC.

\end{cor}

We finish this section with some open questions. First, we have seen in Proposition \ref{nonexistence_projections_ck} that the hereditarily Lindeloff property is sufficient to eliminate uncountable C*-polyhedrons of projections. We can ask what is the situation on general C*-polyhedrons:

\begin{quest}\label{question_lindeloff}
Let $K$ be a hereditarily Lindeloff compact Hausdorff space. Is it true that $C(K)$ does not have an uncountable C*-polyhedron?
\end{quest}
Another question is the relation of polyhedron and C*-polyhedron:

\begin{quest}\label{poly_c_poly}Let $\A$ be a C*-algebra. Is it true that $\A$ has a polyhedron of size $\kappa$ if, and only if, $\A$ has a C*-polyhedron (\cc ) of size $\kappa$ ?
\end{quest}

Observe that, if Question \ref{question_lindeloff} has a positive answer, then the Double arrow space (see Remark \ref{double_arrow}) would give a counter-example to Question \ref{poly_c_poly}. 

As we will see in the next section, Question \ref{poly_c_poly} has a negative answer in the noncommutative context.

\section{Noncommutative C*-algebras and C*-polyhedrons }\label{sec_noncommutative}
In this section, we consider the existence of uncountable C*-polyhedron on classical examples of noncommutative C*-algebras.
Let's begin with the C*-algebra $K(H)$ of all compact operators on some Hilbert space $H$.

\begin{lem}\label{unitization}Let $\A$ be a nonunital C*-algebra and $\tilde{\A}$ its unitization. If $\tilde{\A}$ has an uncountable C*-polyhedron (\cc), then $\tilde{\A}$ has one which is formed by elements from $\A$.
\end{lem}
\begin{proof}
Suppose $\mathcal{F}$ is an uncountable C*-polyhedron in $\tilde{\A}$. 

If $T\in \tilde{\A}$ and $T \not \in \A$, then $T = S + \lambda 1$, where $S\in \A$  and $\lambda \in \mathbb{C}$. By Lemma \ref{replace_elements}, every element $T = S + \lambda 1 \in \mathcal{F}$ can be replaced by some element of the form $T' = S + \lambda' 1$, where $\lambda'$ belongs to some countable dense set of $\mathbb{C}$. In particular, we can find an uncountable C*-polyhedron $\mathcal{F}'$ in $\tilde{\A}$ such that every element $T\in \mathcal{F}'$ is of the form $T = S + \lambda 1$, where $S\in \A$ and $\lambda\in \mathbb{C}$ is fixed. 
Then $\mathcal{F}'' = \mathcal{F}' - \lambda 1$ is an uncountable C*-polyhedron in $\tilde{\A}$ with elements in $\A$.
\end{proof}

In particular, by Lemma \ref{unitization}, makes sense to talk about C*-polyhedron or \cc in nonunital C*-algebra. 

\begin{prop}\label{polyhedron_finite_dimension}Fix $H$ a Hilbert space and let $K(H)$ be the C*-algebra of all compact operators on $H$. Let $\A = K_1(H)$ be the C*-algebra generated by $K(H)$ and the identity 1 of $B(H)$. Then $\A$ has an uncountable C*-polyhedron (\cc)  if and only if $\A$ has an uncountable C*-polyhedron (\cc)  formed by finite rank operators. Moreover, we can assume that there is $n\in \mathbb{N}$, such that every operator has rank equal to $n$.
\end{prop}

\begin{proof}
Suppose $\A$ has an uncountable C*-polyhedron (\cc). By Lemma \ref{unitization}, we can assume that $\A$ has an uncountable C*-polyhedron (\cc) of compact operators.
Since finite rank operators are dense in $K(H)$, the result follows from  Lemma \ref{replace_elements}.

\end{proof}

\begin{defn}Let $H$ be a Hilbert space. Given $v, w\in H$, denote by $\omega_v$ the vector state on $H$ defined by $\omega_v(x) = \langle x, x\rangle$ and denote by $v\otimes w$ the one-dimensional projection defined by $v\otimes w (x) = \langle x, v\rangle w$.

\end{defn}

\begin{thm}\label{non_exist_compact}Let $\A = K(H)$. Then $\A$ does not have an uncountable \cc.

\end{thm}

\begin{proof}Suppose $\A$ has an uncountable \cc and lets get a contradiction. Fix $(a_\alpha)_{\alpha<\omega_1}$ an uncountable \cc.

By Lemma \ref{polyhedron_finite_dimension}, we can assume that each $a_\alpha$ is finite-dimensional, with dimensional equals to some $d$. Let's write 
$$a_\alpha = \sum_{j=1}^d \lambda_{\alpha, j} f_{\alpha, j}\otimes f_{\alpha, j}$$
where $(f_{\alpha, j})_{1\leq j \leq d}$ is an orthonormal basis for the range of $a_\alpha$ and $f_{\alpha, j}\otimes f_{\alpha, j}$ is the rank-one projection onto $\overline{span\{f_{\alpha, j}\}}$.

Since  $\{(\lambda_{\alpha, 1}, \cdots \lambda_{\alpha, d}): \alpha<\omega_1\}$ is an uncountable subset of $\mathbb{C}^d$, there is $\alpha$ and a sequence of index $(\beta_n)_{n\in \mathbb{N}}$ such that $\beta_n \neq \alpha$ for each $n\in \mathbb{N}$ and \footnote{Otherwise, we can find an uncountable discrete set in $\mathbb{C}^n$.}
$$(\lambda_{\alpha, 1}, \cdots \lambda_{\alpha, d})\in \overline{\{(\lambda_{\beta_n, 1}, \cdots \lambda_{\beta_n, d}): n\in \mathbb{N}\}}.$$

For each $n \in \mathbb{N}$, consider the compact operator $T_n \in \A$ defined by 
$$T_n = \sum_{j=1}^d f_{\alpha, j} \otimes f_{\beta_n, j}.$$
Then 
${T^*_n = \sum_{j=1}^d f_{\beta_n, j} \otimes f_{\alpha, j}}$
and $T^*_n T_n\leq I$.

%\begin{align*}
%& T^*_mT_m(x) & = &T^*_m(\sum_{j=1}^d \langle x, f_{\alpha, j} \rangle f_{\beta_m, j}) \\
%& &=&\sum_{j=1}^n \langle x, f_{\alpha, j} \rangle T^*_m(f_{\beta_m, j}) \\
%& & = &  \sum_{j=1}^d \langle x, f_{\alpha, j} \rangle \sum_{i=1}^d \langle f_{\beta_m, j}, f_{\beta_m, i} \rangle f_{\alpha, i} \\
%& &=& \sum_{j=1}^d \langle x, f_{\alpha, j}\rangle f_{\alpha, j} \\
%& & =& P_{{spam\{f_\alpha, j}, 1\leq j\leq n\}}(x) 
%\end{align*}
Let us prove that $(T^*_n a_{\beta_n}T_n)_{n\in \mathbb{N}}$ converges to $a_\alpha$. 
Fix $\varepsilon>0$ and let $n\in \mathbb{N}$ such that

$$\max_{1\leq i\leq d} |\lambda_{\alpha, i} - \lambda_{\beta_n, i}|<\frac{\varepsilon}{d}$$
Observe that
%\begin{align*}
%& T^*_m a_{\beta_m} T_m  & =  (\sum_{j=1}^d f_{\beta_n, j} \otimes f_{\alpha, j})(\sum_{j=1}^d \lambda_{\beta, j} f_{\beta, j}%\otimes f_{\beta, j})(\sum_{j=1}^d f_{\alpha, j} \otimes f_{\beta_n, j})\\
%& & =    \sum_{j=1}^d \lambda_{\beta_m, j} f_{\alpha, j}\otimes f_{\alpha, j}
% \end{align*}
 \begin{align*}
T^*_n a_{\beta_n} T_n (x) & =   T^*_n a_{\beta_n} (\sum_{j=1}^d \langle x, f_{\alpha, j} \rangle f_{\beta_n, j})  \\
							&= T^*_n (\sum_{j=1}^d \langle x, f_{\alpha, j} \rangle a_{\beta_n} (f_{\beta_n, j})) \\
 							& =T^*_n (\sum_{j=1}^d \langle x, f_{\alpha, j} \rangle \lambda_{\beta_n, j} f_{\beta_n, j}) \\ 
							& =(\sum_{j=1}^d \langle x, f_{\alpha, j} \rangle \lambda_{\beta_n, j} T^*_n(f_{\beta_n, j})\\
							& = \sum_{j=1}^d \langle x, f_{\alpha, j} \rangle \lambda_{\beta_n, j} \langle f_{\beta_n, j}, f_{\beta_n, j} \rangle f_{\alpha, j}\\
							& = \sum_{j=1}^d \lambda_{\beta_n, j} \langle x, f_{\alpha, j} \rangle f_{\alpha, j}\\
							& = \sum_{j=1}^d \lambda_{\beta_n, j} f_{\alpha, j}\otimes f_{\alpha, j}(x)
\end{align*}
Then
\begin{align*}
\|(a_{\alpha} - T^*_n a_{\beta_n} T_n)\| & =  \|\sum_{j=1}^d \lambda_{\alpha, j} f_{\alpha, j}\otimes f_{\alpha, j} - \sum_{j=1}^d \lambda_{\beta_n, j} f_{\alpha, j}\otimes f_{\alpha, j})\|\\
 & = \|\sum_{j=1}^d (\lambda_{\alpha, j} - \lambda_{\beta_n, j}) f_{\alpha, j}\otimes f_{\alpha, j} \|\\
 &\leq   \sum_{j=1}^d |\lambda_{\alpha, j} - \lambda_{\beta_n, j}|\| f_{\alpha, j}\otimes f_{\alpha, j}\|\\
 &\leq \sum_{j=1}^d|\lambda_{\alpha, j} - \lambda_{\beta_n, j}| \\
 & \leq  d \max_{1\leq i\leq d} |\lambda_{\alpha, i} - \lambda_{\beta_n, i}|<\varepsilon
\end{align*}

Since $\varepsilon$ was arbitrary, we have proved that  $(T^*_n a_{\beta_n} T_n)_n$ converges to $a_\alpha$. 
To conclude, observe that $(T^*_n a_{\beta_n} T_n)_n$  is a sequence of elements in $MCL(\{a_{\beta_n}:n\geq 1\}\cup \{0\})\subset MCL(\{a_\beta: \beta \neq \alpha\}\cup\{0\})$. In particular
$$a_\alpha \in MCL(\{a_\beta: \beta \neq \alpha\}\cup\{0\})$$
which is a contradiction with the fact that $(a_\alpha)_{\alpha<\omega_1}$ is a \cc.
\end{proof}

\begin{cor}For every infinite cardinal $\kappa$, there is a C*-algebra $\A$ of density $\kappa$ such that $\A$ has only countable \cc.
\end{cor}
\begin{proof}
For every infinite cardinal $\kappa$, let $\A = K(\ell_2(\kappa))$. Then $\A$ has density $\kappa$ and by Theorem \ref{non_exist_compact}, $\A$ does not have uncountable \cc.
\end{proof}

The following remark is a counter-example to Question \ref{question_existence} and Question \ref{poly_c_poly} :

\begin{rem}Let $H = \ell_2(\omega_1)$ and $\A = K(H)$. Then $\A$ is a nonseparable C*-algebra ($d(\A) = \omega_1$) and by Theorem \ref{non_exist_compact}, $\A$ does not have an uncountable \cc. In particular, $\A$ is an example of a nonseparable C*-algebra without uncountable \cc. On the other hand, consider $(e_\alpha)_{\alpha<\omega_1}$ an orthogonal basis for $H$ and $(P_\alpha)_{\alpha<\omega_1}$ the corresponding rank-one projections. Finally, define $\tau_\alpha$ to be the vector state associated with $e_\alpha$. Then $(P_\alpha, \tau_\alpha)_{\alpha<\omega_1}$ is a UBABS on $\A$ and by \cite[Proposition 2.2]{granero2003kunen},  we conclude that $(P_\alpha)_{\alpha<\omega_1}$ is an uncountable polyhedron in $\A$. 

\end{rem}

%Now we turn our attention to $B(H)$:

%Facts: 
%\begin{enumerate}
%\item Self-adjoint elements with finite spectrum are dense in the set of all self-adjoint elements in $B(H)$.
%(C-convexity and matricial ranges - Farenick Lemma 1)
%\item Let $T$ be a self-adjoint element with finite spectrum. Then every element in the spectrum is an eigenvalue %and $H$ has an orthonormal basis consisting of eigenvectors of $T$.
%\end{enumerate}

% https://math.stackexchange.com/questions/2285258/equivalent-condition-of-unitary-equivalence-of-projections-in-a-hilbert-space

\begin{prop}\label{bh_non_compact}Let $\A = B(H)$. Then $\A$ does not have an uncountable C*-polyhedron of compact operators.

\end{prop}

\begin{proof}
Suppose $(a_\alpha)_{\alpha<\omega_1}$ is an uncountable C*-polyhedron of compact operators in $B(H)$. By \cite[Lemma 6]{loebl1981some}, every C*-convex set in $B(H)$ which contains a compact operator, necessarily contains $0$. In particular, for every $\alpha<\omega_1$
$$MCL(\{a_\beta: \beta\neq \alpha\}) = MCL(\{a_\beta:\beta\neq \alpha\}\cup \{0\})$$
and therefore, $(a_\alpha)_{\alpha<\omega_1}$  would be a \cc in $B(H)$, hence, in $K(H)$, which is a contradiction with Theorem \ref{non_exist_compact}.
\end{proof}

\begin{prop}\label{bh_non_existence}$B(\ell_2)$ has no uncountable C*-polyhedron of self-adjoint elements.
\end{prop}
\begin{proof}
The proof has similar ideas with the proof of Theorem \ref{non_exist_compact}.
Suppose $B(\ell_2)$ has an uncountable \cc of self-adjoint elements and let's get a contradiction.

Suppose $(a_\alpha)_{\alpha<\omega_1}$ is a \cc. Since self-adjoint elements with finite spectrum are dense in the set of all self-adjoint elements in $B(\ell_2)$ (see ) we can assume (by Lemma \ref{replace_elements}) that each $a_\alpha$ is a self-adjoint element with finite spectrum.
In particular, we can write %\footnote{https://math.stackexchange.com/questions/3799698/self-adjoint-bounded-operator-with-finite-spectrum-implies-diagonalisable}
$$a_\alpha = \lambda_{\alpha, 1} P_{\alpha, 1} + \cdots \lambda_{\alpha, n} P_{\alpha, n_\alpha}$$
where $\lambda_{\alpha, 1}, \cdots, \lambda_{\alpha, n_\alpha}$ are eigenvalues of $a_\alpha$ and $(P_{\alpha, i})_{1\leq i\leq n_\alpha}$ are pairwise orthogonal projections.
By a counting argument, we can assume that $n_\alpha = d$ and  $dim(Range(P_{\alpha, i})) = dim(Range(P_{\beta, i}))$ for each $i$ and $\alpha, \beta<\omega_1$.

Since  $\{(\lambda_{\alpha, 1}, \cdots \lambda_{\alpha, d}): \alpha<\omega_1\}$ is an uncountable subset of $\mathbb{C}^d$, there is $\alpha$ and a sequence of index $(\beta_n)_{n\in \mathbb{N}}$ such that $\beta_n \neq \alpha$ for each $n\in \mathbb{N}$ and 
$$(\lambda_{\alpha, 1}, \cdots \lambda_{\alpha, d})\in \overline{\{(\lambda_{\beta_n, 1}, \cdots \lambda_{\beta_n, d}): n\in \mathbb{N}\}}.$$

Fix $\varepsilon>0$ and let $n\in \mathbb{N}$ such that
$$\max_{1\leq i\leq d} |\lambda_{\alpha, i} - \lambda_{\beta_n, i}|<\frac{\varepsilon}{d}$$
Let $(e_{\delta, i, j}: j\in \mathbb{N})$ be an orthonormal basis for $Range(P_{\delta, i})$  for $\delta=\alpha, \beta_n$ and define $A_{i}(e_{\alpha, i, j}) = e_{\beta_n, i, j}$.

If $A = \sum_{i=1}^d A_i$ then $A^* A\leq I$ and 
\begin{align*}
\|a_\alpha - A^* a_\beta A\|  & =  \|\sum_i^d \lambda_{\alpha, i} P_{\alpha, i} - A^* ( \sum_i^d \lambda_{\beta, i} P_{\beta, i} )A\|\\
                              & = \|\sum_i^d \lambda_{\alpha, i} P_{\alpha, i} - \sum_i^d \lambda_{\beta, i} A_i^* P_{\beta, i} A_i\| \\
                              & = \|\sum_i^d (\lambda_{\alpha, i}  -\lambda_{\beta, j} )P_{\alpha, i}\|\\
                              & \leq \sum_i^d |\lambda_{\alpha, i}  -\lambda_{\beta, j}|\\
                              &\leq \varepsilon
\end{align*}
This proves that $a_\alpha \in MCL(\{a_{\beta_n}:\}\cup\{0\})\subset MCL(\{a_\beta: \beta\neq \alpha\}\cup\{0\})$, which is a contradiction with the fact that $(a_\alpha)_{\alpha<\omega_1}$ is a \cc.
%\begin{lem}$K(H)$ has a noncommutative polyhedron of size $den(H)=den(K(H))$.
%\end{lem}
%\begin{proof}
%Let $(e_\xi)_{\xi<\kappa}$ be an orthonormal basis of $H$, and $P_\xi$ the projection onto $e_\xi$.
%\end{proof}
\end{proof}

%\bibliography{/home/giba/Dropbox/Estudos/Bibliografia/bibliografia.bib}{}  % HP
%\bibliography{/home/csh/Dropbox/Estudos/Bibliografia/bibliografia.bib}{}  % DELL % associado ao arquivo: 'bibliografia.bib'
%\bibliography{/home/clayton/Dropbox/Estudos/Bibliografia/bibliografia.bib} % Lenovo
\bibliographystyle{abbrv}

\end{document}